\documentclass[12pt,a4paper]{amsart}
\usepackage{amsmath,amsthm,amsfonts,amscd,amssymb,eucal,latexsym,mathrsfs}
\usepackage[all]{xy}
\usepackage{latexsym, amssymb, amscd, mathrsfs, graphics, graphicx, array}

\theoremstyle{plain}
\newtheorem{theorem}{Theorem}

\newtheorem{lemma}[theorem]{Lemma}
\newtheorem{proposition}[theorem]{Proposition}

\theoremstyle{definition}
\newtheorem{remark}[theorem]{Remark}

\newtheorem{problem}[theorem]{Problem}
\newtheorem{definition}[theorem]{Definition}

\newcommand{\acts}{\curvearrowright}

\newcommand{\sF}{{\mathscr F}}

\newcommand{\G}{\Gamma}

\newcommand{\rk}{\mathrm{rk}}

\usepackage{tikz}

\usetikzlibrary{calc}

\date{\today}

\title[Tits buildings and groups with the Haagerup property]{An exotic group with the Haagerup property} 

\author{Sylvain Barr\'e}
\address{\hskip-\parindent
Sylvain Barr\'e, Universit\'e de Bretagne Sud, Universit\'e Europ\'eenne de Bretagne, France}
\email{sylvain.barre@univ.ubs.fr}
\author{Mika\"el Pichot}
\address{\hskip-\parindent
Mika\"el Pichot, Dept. of Mathematics \& Statistics, McGill University, Montr\'eal, Quebec, Canada H3A 2K6}
\email{pichot@math.mcgill.ca}

\begin{document}

\begin{abstract}
We prove the Haagerup property for an infinite discrete group constructed using surgery on a Euclidean Tits building of type $\tilde A_2$. 
\end{abstract}

\maketitle

\bigskip

The group $\G_{\bowtie}$ studied in this paper is the fundamental group of a 2-dimensional cell complex $V_{\bowtie}$ defined by gluing  13 faces along their boundaries respecting orientation and labelling as follows:

\begin{figure}[h]
\centering{\includegraphics[width=10cm]{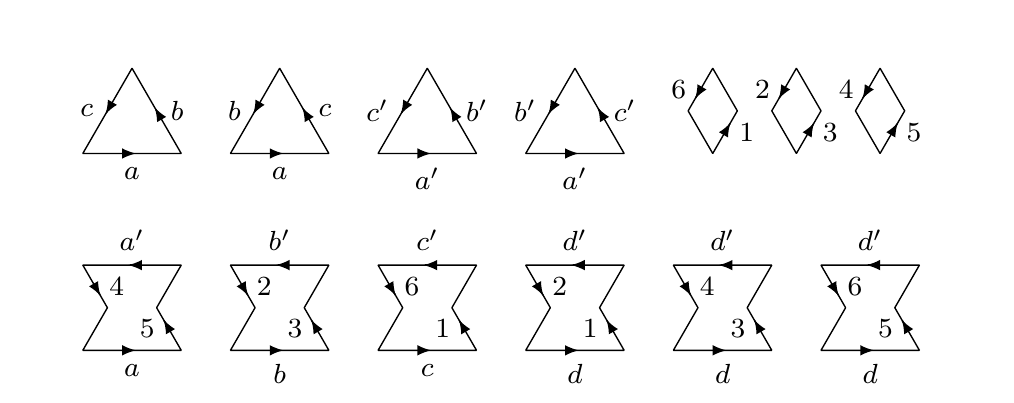}}
\caption{}\label{bowtie}
\end{figure}

\noindent Each face is Euclidean (flat) and the universal cover $X_{\bowtie}$ of $V_{\bowtie}$ is a CAT(0) space  in the induced metric \cite[Lemma 62]{rd}. The metric on the faces is given by their embedding in the Euclidean plane described on Figure \ref{bowtie}. On the first line, the first four faces are equilateral triangles and the next three are rhombi. We call the six lower faces \emph{bow ties}.

We refer to \cite{rd} for information on the group $\G_{\bowtie}$ (see in particular Definition 61 therein). The construction of  $V_{\bowtie}$  proceeds in several steps  using a surgery operation  on a Euclidean (exotic) Tits buildings of type $\tilde A_2$ discovered in \cite[Section 3]{toulouse}. In order to create the above presentation of $V_{\bowtie}$, one starts by cutting the  building into two parts, chooses  (appropriately) one of the two components, and glues it to a duplicate of itself along their common boundary.

In the present paper we prove that:

\begin{theorem}\label{th1}
The group $\G_{\bowtie}$ has the Haagerup property.
\end{theorem}

The proof of this result is based  on the construction of a wall space in $X_{\bowtie}$, in the sense of Haglung and Paulin \cite{HP}, which is proper; this implies that the group $\G_{\bowtie}$ has the Haagerup property  (using the usual affine representation of Haglung, Paulin and Valette). In fact we will see, more precisely, that the group $\G_{\bowtie}$ acts properly on a CAT(0) cube complex of dimension 4.

We construct the walls in Section \ref{S-walls}, prove that they are separating in Section \ref{S-Sep}, and prove properness of the wall space in Section \ref{S-prop}. 
For more details on the Haagerup property (which is also known as Gromov's a-T-menability), we refer to \cite{Valette-book}.

\begin{remark}
Although not exactly a Tits building, the complex $X_{\bowtie}$ has links isomorphic to spherical buildings (of type $A_2$) at infinitely many of its vertices. The surgery procedure mentioned above avoids the vertices of the original Euclidean building. In fact,  
\begin{itemize}
\item[-] the link at every vertex of every equilateral triangle in $X_{\bowtie}$, 
\item[-] the link at every acute vertex of every rhombus in $X_{\bowtie}$, and 
\item[-] the link at every acute vertex of every bow tie in $X_{\bowtie}$,
\end{itemize}
is a  spherical building, namely, the incidence graph of the Fano plane.  (The fact that the space $V_{\bowtie}$ is  locally CAT(0) is proved in Lemma 62 of \cite{rd}, which describes the link at the other vertices as well.) The group $\G_{\bowtie}$ was an important example of a group ``of intermediate rank" in \cite{rd}. We recall that (by a result of Tits, see e.g.\ \cite{Tits81}) every simply connected simplicial complex whose links at {\bf every} vertex are spherical buildings of type $A_2$ is itself a Euclidean building of type $\tilde A_2$. Of course, it is well known that a group which acts properly and uniformly on such a complex has the property T of Kazhdan (see \cite{BHV}).
\end{remark}

\bigskip

\centerline{---}

\medskip

Let us  give additional motivation for Theorem \ref{th1} in relation with \cite{rd}. The paper \cite{rd}  provides a new point of view on the notion of rank for groups and spaces.  The rank of a group is a useful invariant (for example, it is useful to understand the rigidity properties of the group)  but it is not clear how to it can be defined  beyond the realm of Lie groups and their lattices. The idea in \cite{rd} is to ``interpolate" between situations of rank 1 and situations of  rank $\geq 2$ (the ``higher rank" lattices, especially in the irreducible case), yielding intermediate values for the ``rank" (in a sense to be made precise) of the resulting groups and spaces. For example, we introduced there groups ``of rank $7\over 4$", as well as the group $\G_{\bowtie}$ acting on $X_{\bowtie}$. We were especially interested in constructing groups whose ``rank"  is close (but unequal) to 2.  The question of rigidity, for example, (e.g.\ property T but also Mostow--Margulis rigidity, QI rigidity, geometric superrigidity via harmonic maps, etc.) is  interesting in this intermediate rank environment. 

A quantitative  notion of rank for cocompact actions  $\G\acts X$ of a group $\G$ on singular CAT(0) 2-complexes $X$, denoted $\rk(X,\G)$ and called the local rank, was made explicit later in \cite{chambers} (see Definition 4.5 therein). Using this definition, one finds that the local rank of the group $\G_{\bowtie}$ at $X_{\bowtie}$ equals:
\[
\rk(X_{\bowtie},\G_{\bowtie})={31\over 16}\simeq 1.9375.
\]

 In \cite{random}, we constructed (random) groups and spaces whose local rank $\rk(X,\G)$ is as close to 2 as desired.   These groups, generically, have the property T of Kazhdan (and they act on CAT(0) complexes built out of equilateral triangles). The fact that the local rank approaches 2 while the local geometry is kept uniformly bounded forces the apparition of (many) links isomorphic to spherical buildings in the underlying complex $X$. 
Theorem \ref{th1} (and its analogous result for the groups in $\sF_{\bowtie}$ described in  Section \ref{add-rem}) is one step towards understanding the properties of this type of groups of ``high rank".  It is a priori rather surprising that a group like $\G_{\bowtie}$ enjoys the Haagerup property. 
 
 We haven't been able to resolve the following problem (which is interesting both when the local geometry is uniformly bounded,  and in the general case):

\begin{problem}[High rank versus the Haagerup property]
Is there a real number $r_{\mathrm{crit}}<2$ (either absolute or depending only on a uniform bound on the local geometry, that is, on a uniform upper bound on the number of triangles containing any given point) such that if $(X,\G)$ is a couple for which
\[
\rk(X,\G)>r_{\mathrm{crit}}
\]
where $\G$ acts freely isometrically on a singular CAT(0) triangle complex $X$ with compact quotient, then the group $\G$ does not have the Haagerup property?
\end{problem}

The  underlying idea conveyed by the problem  is that the presence of a high proportion of ``good" expanders (say, \emph{Ramanujan graphs}) in the local geometry of $X$ might prevent the group $\G$ to have the Haagerup property, or at least,   in principle, the space $X$ to carry a \emph{separating} proper wall space.  More details on this idea are provided in Section \ref{S-Sep} and Remark \ref{Rem - ramanujan} below.   It would be interesting to study the above problem for other types of Euclidean buildings (insofar as the equality $\rk(X,\G)=2$ implies that $\G$ has Kazhdan's property T).

We note that $r_{\mathrm{crit}}\geq {31\over 16}$ by Theorem \ref{th1}. 
\bigskip

\emph{Acknowledgments.} This work is partially supported by NSERC and FQRNT.

\section{The Walls}\label{S-walls}

The complex $X_{\bowtie}$ has five types of walls:  walls of type $x$ where $x=a,b,c,d$, and transverse walls. In this section we describe these walls.
They are defined by the footprints they leave on faces.

\begin{definition}\label{D - Wall a}
A \emph{wall of type $a$} is a connected component of the subcomplex of $X_{\bowtie}$ defined by taking the universal cover of the set of (red)  edges and rhombi of $V_{\bowtie}$ shown on Fig. \ref{fig - walls of type a}.
\end{definition}

\begin{figure}[h]
\centering{\includegraphics[width=10cm]{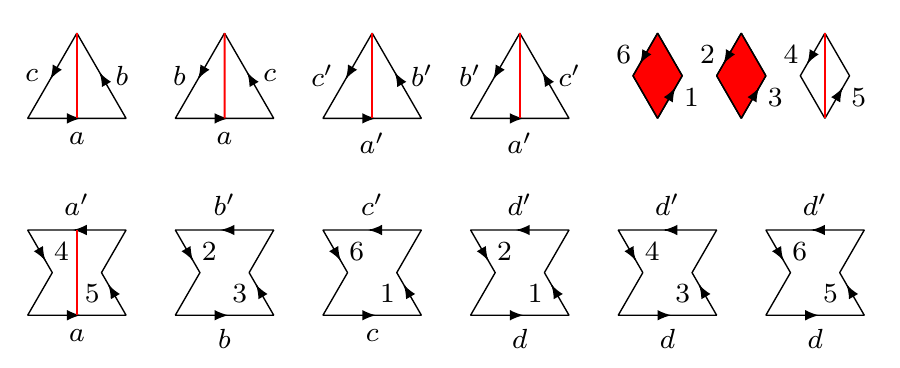}}
\caption{Walls of type $a$}\label{fig - walls of type a}
\end{figure}

 We define similarly the walls of type $b$ and the walls of type $c$ by permuting appropriately the labels on faces.
To illustrate, Fig. \ref{fig - walls of type b} below describes the walls of type $b$.

\begin{figure}[h]
\centering{\includegraphics[width=10cm]{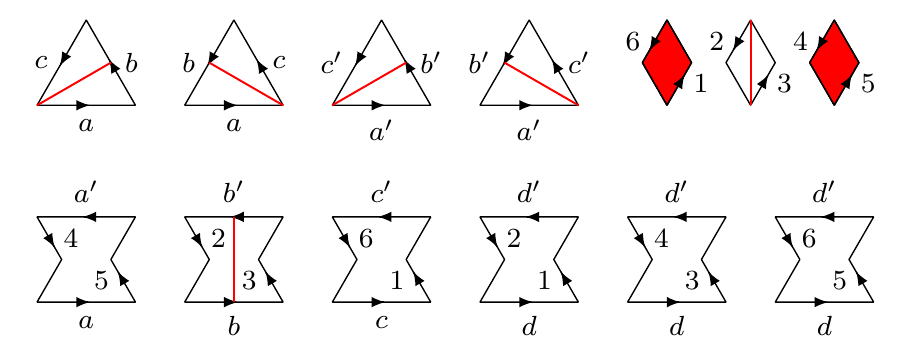}}
\caption{Walls of type $b$}\label{fig - walls of type b}
\end{figure}

\begin{remark}\label{Rem - ramanujan}
There are several difficulties with these walls. The major difficulty is that they cross spherical buildings at vertices where, in particular,  the first eigenvalue is greater than $1/2$  (see Section 5 in \cite{BHV}, in particular Section 5.5 to 5.7 for  the role of the condition  ``the first eigenvalue is greater than $1/2$"). In fact, the link at {\bf every} vertex in such a wall  is a spherical building of type $A_2$, namely, the incidence graph of the Fano plane. The spectral properties of this graph make it a ``good" expander --- it is a Ramanujan graph --- and yet, it will have  along the proof to be cut    \emph{into two large subsets without discarding too many edges}. Another difficulty is that the walls  `refract' (see below) and bounce along CAT(0) geodesics in $X_{\bowtie}$ (see Section \ref{S-prop}).  
\end{remark}

\begin{definition}
A \emph{wall of type $d$} is a connected component of the graph defined by taking the universal cover of the  set of (red)  edges of $V_{\bowtie}$ described on Fig. \ref{Fig - walls of type d}. They enter no face with shapes or labels  other than the ones in Figure \ref{Fig - walls of type d}. 
\end{definition}

\begin{figure}[h]
\centering{\includegraphics[width=5.5cm]{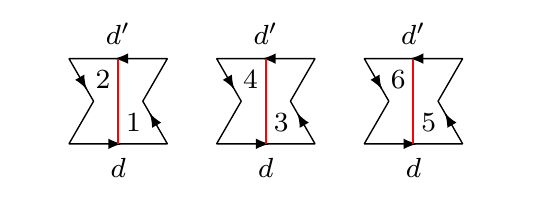}}
\caption{Walls of type $d$}\label{Fig - walls of type d}
\end{figure}

\begin{definition}
A \emph{transverse wall} in $X_{\bowtie}$ is a connected component of the graph defined by taking the universal cover of the set of (red)  edges of $V_{\bowtie}$ shown on Fig. \ref{Fig - transverse walls}. Transverse walls enter no triangular face.
\end{definition}

\begin{figure}[h] 
\centering{\includegraphics[width=10cm]{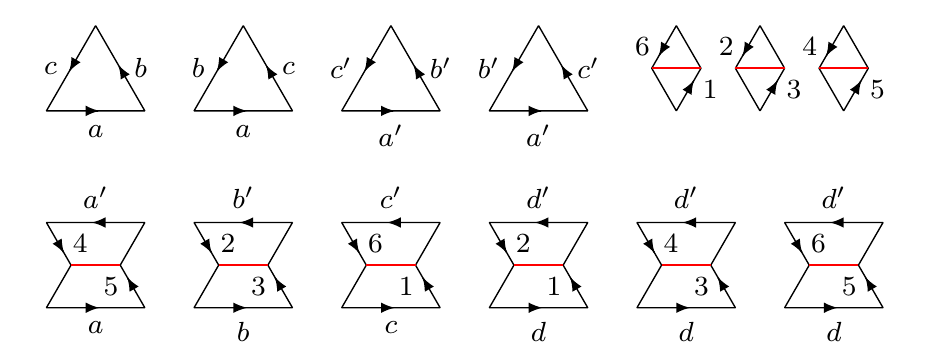}}
\caption{Transverse walls}\label{Fig - transverse walls}
\end{figure}

In the next section, we define and study the wall space of all connected components of complements of the above geometric walls in $X_{\bowtie}$.

\begin{lemma}\label{L - walls - simply connected}
The walls are simply connected.
\end{lemma}

\begin{proof}
Let $W$ be a wall of type $x=a,b,c$ and assume that we have a cycle $C$ (of minimal length)  inside $W$. Let $D$ be a disk diagram (of minimal area) filling $C$.  If $D$ contains no bow tie, then  we have an immediate contradiction from the fact that $C$ would be locally geodesic. Otherwise, let $B$ be a bow tie inside $D$. It is easy to see from the presentation of $V_{\bowtie}$ that $D$ then contains a strip $S$ of alternating rhombi and bow ties extending $B$ in both direction up to its boundary $C$.  Denote by $B_1$ and $B_2$ the two bow ties lying on $C$ and  corresponding to this strip. We may choose $B$ in such a way that there is a path in $C$ between $B_1$ and $B_2$ which contains no bow tie other than $B_1$ and $B_2$. It follows that the wall is locally geodesic except perhaps for two refraction points (see Figure \ref{Fig - refraction}) at two vertices of $B_1$ and $B_2$. This again contradicts uniqueness of geodesics in CAT(0) spaces.  Thus, $W$ does not contain any non trivial cycle.

The walls of type $d$ and the transverse walls are totally geodesic and thus, simply connected. 
\end{proof}

\begin{figure}[h]
\centering{\includegraphics[width=4cm]{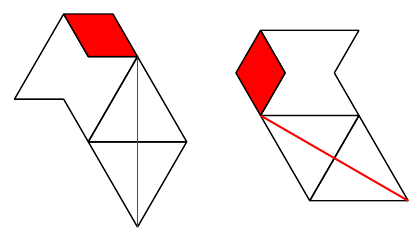}}
\caption{Wall refraction}\label{Fig - refraction}
\end{figure}

\section{The separation}\label{S-Sep}

We show in this section that the wall complements have exactly two connected components.  

\begin{lemma}
For every  wall $W$ of type $x=a,b,c$, the set $X_{\bowtie}\setminus W$ has at least two connected components.
\end{lemma}

\begin{proof}
We assume that $W$ is of type $a$.  Let $e$ be an edge orthogonal to $W$ with label $a$, and let $x$ and $y$ be the extremities of $e$. We show that $x$ and $y$ belong to two distinct components. If not, there is a piecewise linear path $\gamma_1$ from $x$ to $y$ in $X_{\bowtie}$ which does not intersect $W$. Let $(\gamma_t)_{t\in [0,1]}$ be a piecewise linear homotopy in $X_{\bowtie}$ between $\gamma_1$ and  the edge $e$. There exists a largest $0<t_0<1$ such that $\gamma_t$ intersect $W$ for every $t\leq t_0$. Let $z$ be a point in $\gamma_{t_0}\cap W$. From Definition  \ref{D - Wall a}, it is easy to see that $z$ must be a vertex of $X_{\bowtie}$. In addition, the link $L$ of $z$ is the incidence graph of the Fano plane. The path $\gamma_t$, for $t>t_0$ small enough, provides two points $a_t$ and $b_t$ in $L$, which are connected in $L\setminus W$ for $t>t_0$ but are not connected in $L\setminus W$ for $t=t_0$. The left hand side of Fig. \ref{Fig - Halving} shows that this is a contradiction. So $X_{\bowtie}\setminus W$ has at least two connected components. 

The same proof applies to the two other cases. The right hand side of Fig. \ref{Fig - Halving} deals with walls of type $b$, and the case of walls of type $c$ is symmetric. 
\end{proof}

\begin{figure}[h]
\centering{\includegraphics[width=5cm]{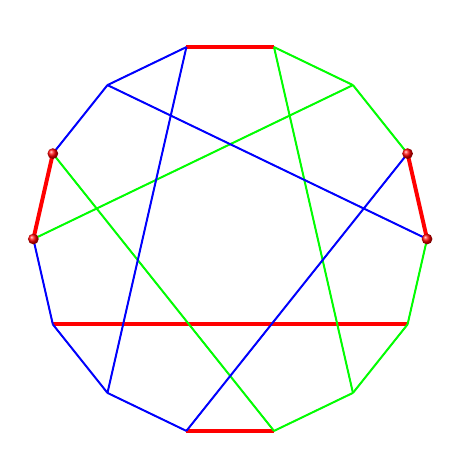}\hspace{1cm}\includegraphics[width=5cm]{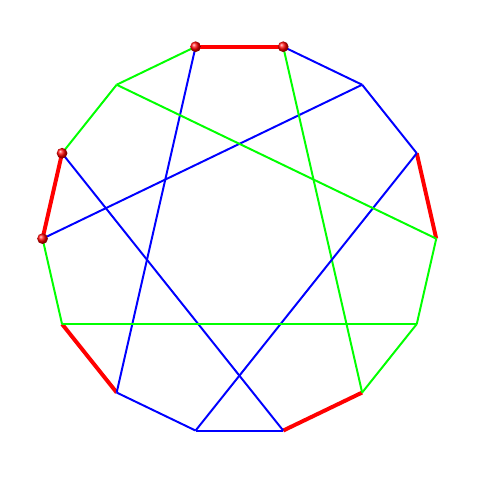}}
\caption{Halving spherical buildings}\label{Fig - Halving} 
\end{figure}

\begin{lemma}
For every  wall $W$ of type $x=a,b,c$, the set $X_{\bowtie}\setminus W$ has at most two connected components.
\end{lemma}

\begin{proof}
Assume that we can find three point  $x,y,z$ in $X_{\bowtie}$ in three distinct components and let us find a contradiction. Let $\Delta$ be the geodesic triangle with vertices $x,y,z$ and let $I$ denote $W\cap \Delta$. By assumption $x,y$ and $z$ are in three distinct components of $\Delta\setminus  I$, so $I$ contains points $x',y',z'$ in the open segments $(y,z)$, $(x,z)$, and $(x,y)$ respectively, and paths $I_x,I_y,I_z\subset I$ joining $y'\leftrightarrow z'$, $x'\leftrightarrow z'$ and $x'\leftrightarrow y'$ respectively.   Either $I_x\cup I_y\cup I_z$ is included in a (red) rhombus, or they form a cycle in the wall $W$.  The first case is easy to rule out and the last case is prohibited by Lemma \ref{L - walls - simply connected}.  
\end{proof}

\begin{proposition}
For every wall $W$, the set $X_{\bowtie}\setminus W$ has exactly two connected components. 
\end{proposition}

\begin{proof}
The case of walls of type $x=a,b,c$ is treated in the above two lemmas, and the two remaining cases are clear. 
\end{proof}

We say that a wall $W$ separates two points $x$ and $y$ in $X_{\bowtie}$ if $x$ and $y$ do not belong to the same connected component of $X_{\bowtie}\setminus W$.

\begin{lemma}
The number of walls through any cell in $X_{\bowtie}$ is (uniformly) finite.  
\end{lemma}

This lemma is obvious --- five  is an upper bound (the number of types of walls).

\begin{proposition}
Given any two points $x,y\in X_{\bowtie}$, the number of walls separating $x$ and $y$ is finite. 
\end{proposition}

\begin{proof}
Every wall  which separates $x$ from $y$ has to intersect the CAT(0) geodesic from $x$ to $y$. Thus the proposition follows from the previous lemma.  
\end{proof}

\section{Properness}\label{S-prop}

Let $\sharp(x,y)$ denote the number of walls separating $x$ from $y$ in $X_{\bowtie}$. The wall space defined in Section \ref{S-walls} is proper whenever
\[
\sharp(x,sx)\to \infty
\] 
for $s\to \infty$ in $\G_{\bowtie}$.

\begin{figure}[h]
\centering{\includegraphics[width=8cm]{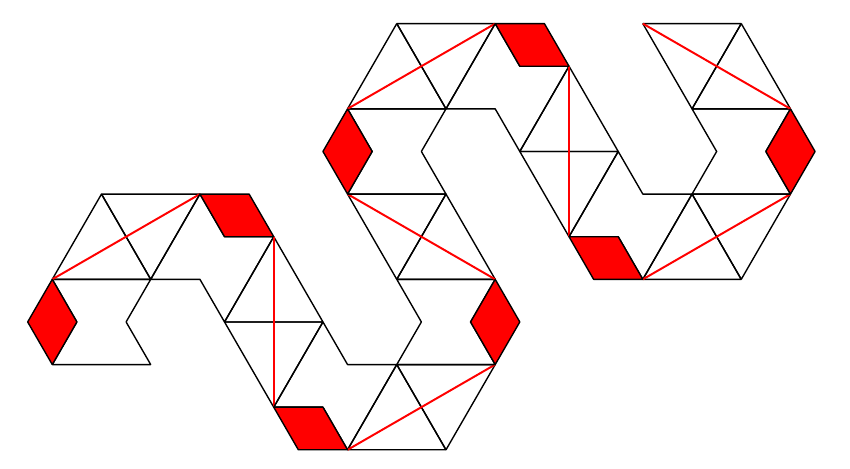}}
\caption{Zigzags in walls}\label{Fig - zigzag}
\end{figure}

Let $d_{\bowtie}(x,y)$ denotes the number of bow ties intersecting the CAT(0)  geodesic $[x,y]$ between $x$ and $y$.
Let $s\to \infty$ in $\G_{\bowtie}$. We distinguish two cases depending on whether $d_{\bowtie}(x,sx)\to \infty$ or not. 
In the first case, we have $\sharp(x,sx)\to \infty$ because every transverse wall intersecting any given bow tie on $[x,y]$  is separating $x$ from $y$. Observe that both the walls of type $d$ and the transverse walls are totally geodesic in $X_{\bowtie}$ (compare to Fig. \ref{Fig - zigzag}, which describes a portion of a wall of type $x=a,b,c$).  

So, we may assume that $d_{\bowtie}(x,sx)$ is bounded above. Since the  length of the geodesic $[x,sx]$ tends to infinity, it follows that an arbitrary long subsegment is contained in a flat plane of equilateral triangles in $X_{\bowtie}$. Thus, properness follows from the following lemma.

\begin{lemma}
If $W$ is a wall of type $x=a,b$ or $c$ and $\Pi$ is a flat plane of equilateral triangles in $X_{\bowtie}$, then $W\cap \Pi$ is  empty or coincide with a straight line in $\Pi$. 
\end{lemma}

\begin{proof}
By definition of $W$, it is clear that $W\cap \Pi$ contains at least a straight line if it is nonempty, and that the intersection consists of parallel straight lines in $\Pi$. Assume that $W\cap \Pi$ contains two such lines $d_1$ and $d_2$.  Using an argument similar to that of Lemma \ref{L - walls - simply connected}, we see that $W$ contain two bow ties $B_1$ and $B_2$ together with two geodesic segments between vertices of $B_1$ and  $B_2$, one inside $W$, and the other one in a strip of bow ties containing $B_1$ and $B_2$. This contradicts the uniqueness of CAT(0) geodesics in $X_{\bowtie}$.  
\end{proof}

Therefore, if $[x,y]$ is a geodesic segment included in a flat plane of equilateral triangles, then $\sharp(x,y)\geq d(x,y)$, where $d$ is the Euclidean distance between $x$ and $y$.  This proves that the wall space is proper.

Theorem \ref{th1} now follows from the combined results in Section \ref{S-Sep}, in Section \ref{S-prop}, and in \cite{HP,Valette-book}.

\section{Additional remarks}\label{add-rem}

(1) The group $\G_{\bowtie}$ belongs to an infinite family $\sF_{\bowtie}$ of groups constructed in a similar way --- using graphs of spaces. The vertex spaces all are isomorphic to the given connected component selected after surgery on the given Euclidean building (see the introduction), and the edge spaces are (nonabelian) free groups. The attaching maps are isometries.  An interesting subfamily is given by gluing $n$ copies of the selected component along a fixed common boundary. It is an easy matter to generalize the idea for constructing walls presented in the present paper to every group in the collection $\sF_{\bowtie}$. 
Every group in $\sF_{\bowtie}$ is built out of 3 shapes (equilateral triangles, and long rhombi and bow ties) and has the Haagerup property.  In fact,   every group in $\sF_{\bowtie}$ admits a proper (isometric) action on a finite dimensional CAT(0) cube complex  (see \cite{S,CN,Nica,HW} for information on group cubulation). The walls that we have described above lead to a CAT(0) cube complex of dimension 4 (indeed, the walls of type $a$, $b$, $c$ and the transverse walls   pairwise intersect each other, but no larger family of walls has this property). Groups in $\sF_{\bowtie}$ provide further examples of groups of intermediate rank in the sense of \cite{rd}. One can explicitly compute their local rank, in the sense of \cite[Definition 4.5]{chambers}, in terms of the underlying graph.

(2) Having the Haagerup property, groups in $\sF_{\bowtie}$ satisfy the Baum--Connes conjecture with coefficients by work of Higson and Kasparov \cite{HK}. On the other hand, the fact that they have a proper action on a finite dimensional CAT(0) cube complex implies that they have the property of rapid decay (see \cite{CR}). The latter property was already established in \cite{rd}, with a different proof (not relying on cube complexes, and which also handles other types of groups of intermediate rank, like the groups of rank $7\over 4$).  Thus, the Baum--Connes conjecture without coefficient was also already known for groups in $\sF_{\bowtie}$.

\end{document}